\documentclass[a4paper,10pt]{article}

\usepackage{color}
\usepackage[english]{babel}
\usepackage{amsmath,amsfonts,amssymb,amsthm}
\usepackage{url}
\usepackage{graphicx}

\newtheorem{theorem}{Theorem}[section]

\newtheorem{lemma}{Lemma}[section]
\newtheorem{proposition}{Proposition}[section]

\theoremstyle{definition}
\newtheorem{remark}{Remark}[section]

\numberwithin{equation}{section}

\newcommand\blfootnote[1]{\begingroup\renewcommand\thefootnote{}\footnote{#1}\addtocounter{footnote}{-1}\endgroup}

\begin{document}

\title{
{\bf\Large Multiple positive solutions \\to elliptic boundary blow-up problems}}

\author{
\vspace{1mm}
\\
{\bf\large Alberto Boscaggin}
\vspace{1mm}\\
{\it\small Dipartimento di Matematica, Universit\`a di Torino}\\
{\it\small via Carlo Alberto 10}, {\it\small 10123 Torino, Italy}\\
{\it\small e-mail: alberto.boscaggin@unito.it}\vspace{1mm}\\
\vspace{1mm}\\
{\bf\large Walter Dambrosio}
\vspace{1mm}\\
{\it\small Dipartimento di Matematica, Universit\`a di Torino}\\
{\it\small via Carlo Alberto 10}, {\it\small 10123 Torino, Italy}\\
{\it\small e-mail: walter.dambrosio@unito.it}\vspace{1mm}\\
\vspace{1mm}\\
{\bf\large Duccio Papini}
\vspace{1mm}\\
{\it\small Dipartimento di Scienze Matematiche, Informatiche e Fisiche, Universit\`a di Udine}\\
{\it\small via delle Scienze 206},
{\it\small 33100 Udine, Italy}\\
{\it\small e-mail: duccio.papini@uniud.it}\vspace{1mm}}

\date{}

\maketitle

\vspace{-2mm}

\begin{abstract}

We prove the existence of multiple positive radial solutions to the sign-indefinite elliptic boundary blow-up problem
$$
\left\{
\begin{array}{ll}
\vspace{0.1cm}
\Delta u + \bigl(a^+(\vert x \vert) - \mu a^-(\vert x \vert)\bigr) g(u) = 0, & \; \vert x \vert < 1,
\\
u(x) \to \infty, & \; \vert x \vert \to 1,
\end{array}
\right.
$$
where $g$ is a function superlinear at zero and at infinity, $a^+$ and $a^-$ are the positive/negative part, respectively, of a sign-changing function $a$ and $\mu > 0$ is a large parameter. In particular, we show how the number of solutions is affected by the nodal behavior of the weight function $a$. The proof is based on a careful shooting-type argument for the equivalent singular ODE problem. As a further application of this technique, the existence of multiple positive radial homoclinic solutions to 
$$
\Delta u + \bigl(a^+(\vert x \vert) - \mu a^-(\vert x \vert)\bigr) g(u) = 0, \qquad x \in \mathbb{R}^N,
$$
is also considered.

\end{abstract}

\section{Introduction}\label{sez-1}

In the qualitative theory of elliptic PDEs, problems of the type
$$
\left\{
\begin{array}{ll}
\vspace{0.1cm}
\Delta u + f(x,u) = 0, & \; x \in \Omega,
\\
u(x) \to +\infty, & \; x \to \partial \Omega,
\end{array}
\right.
$$
where $\Omega \subset \mathbb{R}^N$ is a smooth bounded domain, are usually referred to as \textit{boundary-blow up} problems
and they date back to Bieberback \cite{Bie16} and Rademacher \cite{Rad43}, arising from questions in Geometry and Mathematical Physics.
Further classical contributions were then given by Keller \cite{Kel57}
and Osserman \cite{Oss57}; as for more recent works, we just quote \cite{AftDelLet03,BanMar95,BanMar92,MarVer97}, referring to the introductions of \cite{Gar06,MawPapZan03} for more information 
and references on the subject. 

Our investigation is motivated by a recent paper by Garc{\'{\i}}a-Meli{\'a}n \cite{Gar11}, dealing with the existence of \emph{positive} solutions to
the boundary blow-up problem
\begin{equation}\label{blow-intro2}
\left\{
\begin{array}{ll}
\vspace{0.1cm}
\Delta u + \bigl(\epsilon a^+(x) - a^-(x)\bigr)u^p = 0, & \; x \in \Omega,
\\
u(x) \to +\infty, & \; x \to \partial \Omega,
\end{array}
\right.
\end{equation}
where $1 < p < \tfrac{N+2}{N-2}$, $\epsilon> 0$ is a real parameter and $a^+,a^-$ denote, respectively, the positive and the negative part of a sign-changing function
$a: \overline\Omega \to \mathbb{R}$. Notice that, in view of these assumptions, the equation in \eqref{blow-intro2} is \emph{superlinear indefinite}.
The existence of positive solutions, satisfying Dirichlet or Neumann boundary conditions, has been the object of extensive investigation in the last decades, starting with the pioneering papers \cite{AlaTar93,BerCapNir95} (see also for \cite{FelZan15,FelZanpp} for a wide
bibliography on the subject).
Boundary blow-up conditions, on the other hand, were taken into account in \cite{LG-SMJ-2005} and \cite{LG-JDE-2006} in order to
describe the limit profile of solutions to parabolic problems like:
\[
\begin{cases}
u_{t} - \Delta u = \lambda u + a(x) u^{p} & \text{in } \Omega \times ( 0, +\infty) \\
u = 0                                     & \text{on } \partial \Omega \times ( 0, +\infty ) \\
u(x,0) = u_{0} > 0                        & \text{in } \Omega
\end{cases}
\]
when the weight funtion $ a $ is allowed to change sign in a suitable way inside the domain $ \Omega $.
Roughly speaking, one of the main consequences of the analysis is to outline that the position of $ \lambda $ with respect to the principal
eigenvalues of the different nodal regions of the weight function $ a $ is fundamental to determine the behavior of the solutions.
We address the reader in particular to the very recent monograph \cite{LG-book-2016} for further details and a thorough discussion on the relevance of large solutions in this context.

In \cite{Gar11}, which corresponds to $ \lambda = 0 $ with respect to the discussion carried above, the main result asserts - under some additional technical assumptions - the existence of \emph{two} positive solutions to \eqref{blow-intro2} when $\epsilon > 0$ is sufficiently small.
Incidentally, we notice that, via a standard rescaling, the same result is true for positive blowing-up solutions of the equation
\begin{equation}\label{blow-intro2b}
\Delta u + \bigl(a^+(x) - \mu a^-(x)\bigr) u^p = 0, \qquad x \in \Omega,
\end{equation}
when $\mu > 0$ is large enough.

The aim of the present paper is to show that, at least when $\Omega = B$ is a ball,
a larger number of positive (radial) boundary blow-up solutions to \eqref{blow-intro2b} can arise, depending on the nodal behavior of the weight function.
For instance, as a corollary of our main results we can state the following.

\begin{theorem}\label{th-intro}
Let $B := \{ x \in \mathbb{R}^N \, : \, \vert x \vert < 1\}$ be the unit ball, let $p > 1$ and let $a: [0,1] \to \mathbb{R}$ be a sign-changing continuous function, having finitely many zeros 
in $]0,1[$ and such that $a(0) < 0$, $a(1) < 0$. Then, for $\mu > 0$ large enough the boundary blow-up problem
\begin{equation}\label{blow-intro3}
\left\{
\begin{array}{ll}
\vspace{0.1cm}
\Delta u + \bigl(a^+(\vert x \vert) - \mu a^-(\vert x \vert)\bigr)u^p = 0, & \; x \in B,
\\
u(x) \to +\infty, & \; x \to \partial B,
\end{array}
\right.
\end{equation}
has at least $2^m$ distinct positive radial solutions, where $m$ is the number of disjoint subintervals of $[0,1]$ in which $a$ is non-negative.
\end{theorem}

Notice that the sub-criticality assumption $p < \tfrac{N+2}{N-2}$ here is not necessary. Similar results can be proved even
when $a(0) \geq 0$, requiring however some restrictions on the exponent $p$, depending on the behavior of $a$ near $r = 0$. All this seems to suggest that the role of the critical exponent in superlinear indefinite problems can be quite subtle, thus deserving future investigations. We refer to Section \ref{sez-3} for more general versions of the result, dealing with positive boundary blow-up solutions to $\Delta u + (a^+(\vert x \vert) - \mu a^-(\vert x \vert)) g(u) = 0$
under milder assumptions of $a$ and $g$.

Theorem \ref{th-intro} can be interpreted as the boundary blow-up version of some recent results providing high multiplicity of positive solutions
for boundary value problems associated with superlinear indefinite equations, on a line of research motivated by a conjecture by 
G{\'o}mez-Re{\~n}asco and L{\'o}pez-G{\'o}mez \cite{GomLop00} and later initiated by Gaudenzi, Habets and Zanolin \cite{GauHabZan03}. 
Indeed, in \cite{GauHabZan03} the existence of three positive radial solutions to the Dirichlet problem
$$
\left\{
\begin{array}{ll}
\vspace{0.1cm}
\Delta u + \bigl(a^+(\vert x \vert) - \mu a^-(\vert x \vert)\bigr)u^p = 0, & \; x \in B,
\\
u(x) = 0, & \; x \in \partial B,
\end{array}
\right.
$$
is proved when the weight function $a$ has two intervals of positivity and $\mu > 0$ is sufficiently large. Generalizations were then given
(in various directions, both for ODEs and PDEs, with different growth assumptions and boundary conditions) in \cite{BarBosVer15,BonGomHab05,Bos11,BosFelZan->,FelZan15,FelZanpp,GauHabZan04,GirGom09}
and it is nowadays well understood that $2^m$-multiplicity comes from the possibility of prescribing a-priori (in a singular perturbation
spirit) the behavior of a positive solution in each of the $m$ regions 
of positivity of the weight function: either it is small (actually, arbitrarily close to zero for $\mu \to +\infty$) or it is close to a positive radial solutions to $\Delta u + a^+(\vert x \vert) u^p = 0$ satisfying Dirichlet conditions
at the boundary of the region. Our work thus shows that a similar picture arises when considering the boundary blow-up problem \eqref{blow-intro3}.
 
For the proof of Theorem \ref{th-intro}, we adopt a dynamical systems approach analogous to that employed in \cite{LGTeZa-CPPAA-2014}
where multiple positive solutions are detected for the problem
\[
\begin{cases}
-u'' = \lambda u + a(t) u^{p} & \in (0,1) \\
u(0) = u(1) = M
\end{cases}
\]
where $ M \in \left( 0, +\infty \right] $, $ \lambda < 0 $ and $ a $ is
a symmetric piecewise constant weight function such that
\[
a(t) =
\begin{cases}
- c & \text{for } t \in \left[ 0, \alpha \right) \cup \left( 1-\alpha, 1 \right] \\
b   & \text{for } t \in \left[ \alpha, 1-\alpha \right]
\end{cases}
\]
with $ b, c > 0 $.
There, the multiplicity is mainly a consequence of the fact that for $ \lambda < 0 $ and for $ t \in [ \alpha, 1-\alpha ] $ the equation is
autonomous and has a positive equilibrium which is a local centre.

In our case, we first write (with the usual abuse of notation, $u(x) = u(r)$ for $r = \vert x \vert$) the radial problem associated with \eqref{blow-intro3} as the singular ODE problem
\begin{equation}\label{blowODE}
\left\{
\begin{array}{ll}
\vspace{0.1cm}
\bigl(r^{N-1} u'\bigr)' + r^{N-1}\bigl(a^+(r) - \mu a^-(r)\bigr)u^p = 0, & \; 0 < r < 1,
\\
u'(0) = 0, \quad \lim_{r \to 1^-}u(r) = +\infty, & 
\end{array}
\right.
\end{equation}
and we then develop a shooting-type technique to find positive solutions to \eqref{blowODE}.
More precisely, passing to the phase-plane, we look for intersections between the forward image (from $r = 0$ to $r= 1 - \epsilon$)
of the positive $x$-semiaxis with a planar continuum made by initial conditions (for $r= 1 -\epsilon$) of solutions blowing-up at $r=1^-$.
The construction of such a continuum is borrowed from the paper \cite{MawPapZan03}; here, however, some further localization properties when
varying the parameter $\mu$ are needed.
Instead, the study of the dynamics from $r=0$ to $r=1-\epsilon$ 
is based on an auxiliary lemma, describing the action of the flow map associated with the equation in \eqref{blowODE}
on an interval where $a$ changes sign exactly once.
Differently from \cite{LGTeZa-CPPAA-2014}, in our case the equation does not have a positive equilibrium which the solutions can turn around.
However roughly speaking, in a ``Stretching Along Paths'' spirit (see \cite{DamPap04,MedPirZan09,PapZan02,PaZa-RSMUPT-07}), we are able to show that any path crossing a suitable topological rectangle of the phase-plane gives rise, through the flow map, to $2$ sub-paths with the same property.
Via an iterative use of this result and of the multiple change of sign of the weight function $ a $, we can finally prove that $2^m$ intersections between the image of the positive $x$-semiaxis and the blow-up continuum arise, therefore yielding the existence of $2^m$ positive solutions to \eqref{blowODE}.  

Let us emphasize that this technical lemma seems to be flexible enough to be used in various different situations. For instance, at the end of the paper we will show how it can be applied, together with the Conley-Wa\.zewski's method \cite{Con75,Waz47}, to the search of multiple positive (radial) homoclinic solutions to
\begin{equation}\label{eq-homintro}
\Delta u + \bigl(a^+(\vert x \vert) - \mu a^-(\vert x \vert)\bigr) u^p = 0, \qquad x \in \mathbb{R}^N,
\end{equation}
provided $a$ is negative at infinity (with $\mu > 0$ large enough). 
\medbreak
The plan of the paper is the following. In the final part of this introduction we collect
and comment the main assumptions on the nonlinear term used throughout the paper and we fix some further notation.
Section \ref{sez-2} is devoted to the preliminary topological results: more precisely, in Subsection \ref{sez-2.1} we state and prove the stretching-type lemma,
while in Subsection \ref{sez-2.2} we deal with existence and localization of the continuum of blowing-up solutions.
Then, in Section \ref{sez-3} we state and prove our main results. Finally, in Section \ref{sez-4} we briefly discuss some related results (including the existence
of multiple positive radial homoclinic solutions to \eqref{eq-homintro}) which can be easily obtained with minor modifications of the arguments developed in the paper. 

\subsection{The main assumptions}

As already anticipated, our main results deal with blowing-up radial solutions to 
\begin{equation}\label{eq-hp}
\Delta u + \bigl(a^+(\vert x \vert) - \mu a^-(\vert x \vert)\bigr) g(u) = 0,
\end{equation}
with $g: \mathbb{R}^+ \to \mathbb{R}$ a locally Lipschitz continuous functions satisfying the sign condition
$$
g(0) = 0 \quad \mbox{ and } \quad g(u) > 0, \; \mbox{ for any } \, u > 0, \leqno{(g_*)}
$$
as well as some further superlinearity assumptions.
For convenience, we list and name here some of these hypotheses, since they will be used several times during the paper,
both in connection with \eqref{eq-hp} and with related equations. Precisely, we set
$$
\lim_{u \to 0^+}\frac{g(u)}{u} = 0; \leqno{(g_0)}
$$
$$
\lim_{u \to +\infty}\frac{g(u)}{u} = +\infty; \leqno{(g_{\infty})}
$$
and, defining $G(u) := \int_0^u g(\xi)\,d\xi$ (notice that,
in view of $(g_*)$, $G$ is strictly positive for $u > 0$ and stricly increasing),
$$
\int_1^{+\infty}\frac{d\xi}{\sqrt{G(\xi)}} <+\infty \quad \mbox{ and } \quad
\lim_{u \to +\infty} \int_u^{+\infty}\frac{d\xi}{\sqrt{G(\xi)-G(u)}} = 0. \leqno{(g_{\infty}^*)}.
$$
Some comments are in order. Conditions $(g_0)$ and $(g_\infty)$ just express the fact that $g$ is superlinear
at zero and at infinity, respectively. Instead, hypothesis $(g_\infty^*)$ is typical of boundary blow-up problems.
Precisely, the integrability at infinity of $1/\sqrt{G}$ is the well known Keller-Osserman condition 
(from the pioneering works \cite{Kel57,Oss57}) while the second condition can be interpreted as a time-map assumption for the autonomous ODE
\begin{equation}\label{autonomous}
u'' - g(u) = 0.
\end{equation}
Indeed, it is easily checked that it holds true if and only if
$$
\lim_{c \to -\infty}\sqrt{2}\int_{G^{-1}(-c)}^{+\infty}\frac{d\xi}{\sqrt{G(\xi)+c}} = 0
$$
where the above integral is the total time needed for the solution of (\ref{autonomous}) to run along the orbit
passing through the point $(G^{-1}(-c),0)$.
An extensive discussion about condition $(g_\infty^*)$ can be found in \cite[Appendix]{PapZan00}, where some sufficient conditions for it to hold are
presented, as well. Here, we just recall that $(g_\infty^*)$ is implied by $(g_\infty)$, together with
$$
\int_1^{+\infty}\frac{d\xi}{\sqrt{G(\xi)}} <+\infty \quad \mbox{ and } \quad \liminf_{u \to +\infty}\frac{G(\sigma u)}{G(u)} > 1,  \mbox{ for some } \sigma > 1.
$$
In particular, the model nonlinearity $g(u) = u^p$ satisfies all the above conditions $(g_*)$, $(g_0)$, $(g_\infty)$ and $(g_\infty^*)$
whenever $p > 1$.
\medbreak
\medbreak
\noindent
\textbf{Notation}: For a vector $z = (x,y) \in \mathbb{R}^2$, $\vert z \vert := \sqrt{x^2 + y^2}$ denotes its Euclidean norm. Moreover, 
$Q_i \subset \mathbb{R}^2$ is the (closed) $i$-th quadrant of the plane.

\section{Preliminary results}\label{sez-2}

In this section, we establish the fundamental topological results which will be needed for the proof of our main theorems.
In more detail, Subsection \ref{sez-2.1} is devoted to an auxiliary stretching-type lemma
(to be applied to the study of the dynamics of the equation in \eqref{blowODE} on the interval $[0,1-\epsilon]$),
while in Subsection \ref{sez-2.2} we show existence and localization properties of a continuum of blowing-up solutions.
In view of other possible applications of these results (see Remark \ref{rem-fz} and Section \ref{sez-4})
and in order to keep the exposition to a simpler technical level, 
throughout the section we deal with a regular second order ODE of the type
$$
v'' + q(t)g(v) = 0.
$$
The possibility of applying these results (via a suitable change of variables) 
also to the singular equation in \eqref{blowODE} will be directly discussed along the proof of the main theorems
(see Section \ref{sez-3}).

\subsection{A stretching-type lemma}\label{sez-2.2}

The aim of this section is to prove a stretching-type result for the flow map associated with an equation of the type
$$
v'' + \bigl( b^+(t) - \mu b^-(t)\bigr) g(v) = 0, 
$$
where $b$ is a function which changes sign exactly once.
More precisely, throughout this section we assume that $g: \mathbb{R}^+ \to \mathbb{R}$ is a locally Lipschitz continuous function satisfying $(g_*)$, $(g_0)$ and $(g_\infty)$ and that $b: [\sigma,\omega] \to \mathbb{R}$ is a continuous function such that:
\begin{itemize}
\item [$(b_{*})$]
\textit{there exists $\tau \in \,]\sigma,\omega[$ such that
\begin{align*}
& b(t)\geq 0, \; \text{ for every } t\in [\sigma,\tau], \quad b(t)\not\equiv0 \; \text{ on } [\sigma,\tau]; \\
& b(t)\leq 0, \; \text{ for every } t\in [\tau,\omega], \quad b(t)\not\equiv0 \; \text{ on } [\tau,\omega].
\end{align*}
}
\end{itemize}
To state our result precisely, we define the following extension of $g$,
$$
g_0(v) := g(v^+), \qquad v \in \mathbb{R},
$$
and we consider the equation
\begin{equation}\label{eq-stre}
v'' + \bigl( b^+(t) - \mu b^-(t)\bigr)g_0(v) = 0.
\end{equation}
Let us now observe that, by convexity arguments, local solutions of \eqref{eq-stre} can be continued on the whole $[\sigma,\tau]$; accordingly, we define the 
flow map (notice that such a map is independent on $\mu$, since $b^-(t) = 0$ for $t \in [\tau,\sigma]$)
\begin{equation} \label{def-phi-1}
\varphi_{\sigma,\tau}: \mathbb{R}^2 \to \mathbb{R}^2, \qquad z \mapsto (v(\tau;\sigma,z),v'(\tau,\sigma,z)),
\end{equation}
where $v(\cdot;\sigma,z)$ is the solution of \eqref{eq-stre} satisfying the initial condition 
$$
(v(\sigma;\sigma,z),v'(\sigma;\sigma,z)) = z.
$$ 
Moreover, let also 
$$
\varphi^\mu_{\tau,\omega}: \mathcal{D}_\mu \subset \mathbb{R}^2 \to \mathbb{R}^2, \qquad z \mapsto (v(\omega;\tau,z),v'(\omega;\tau,z)),
$$
where $v(\cdot;\tau,z)$ is the solution of \eqref{eq-stre} satisfying the initial condition 
$$
(v(\tau;\tau,z),v'(\tau;\tau,z)) = z.
$$ 
In the above definition, $\mathcal{D}_\mu \subset \mathbb{R}^2$ is the open set of all the points $ z \in \mathbb{R}^{2} $
such that the solution $v(\cdot;\tau,z)$ exists at least on $[\tau,\omega]$.
Finally, we define
$$
\varphi^\mu_{\sigma,\omega}: \mathcal{D}'_\mu := \varphi_{\sigma,\tau}^{-1}(\mathcal{D}_\mu) \subset \mathbb{R}^2 \to \mathbb{R}^2 \qquad z \mapsto \varphi^\mu_{\tau,\omega}\left( \varphi_{\sigma,\tau}(z)\right).
$$
Of course, $\varphi^\mu_{\sigma,\omega}(z)$ is nothing but the value $(v(\omega),v'(\omega))$ for the solution to
\eqref{eq-stre} satisfying $(v(\sigma),v'(\sigma)) = z$ and,
moreover, $\mathcal{D}'_\mu \subset \mathbb{R}^2$ is an open set, as well.

As a final step, we define, for $0 < r < R$, the sets
\begin{equation}\label{def-ret}
\mathcal{R}(r,R) := \left\{ z \in Q_1 \, : \, \vert z \vert \leq R \right\} \cup \left\{ z \in Q_4 \, : \, \vert z \vert \leq r \right\}
\end{equation}
and
\begin{align*}
\mathcal{R}_{\text{left}}(r,R) & := \{ (0,y) : - r \leq y \leq 0 \} \subset \partial\mathcal{R}(r,R) \\
\mathcal{R}_{\text{right}}(r,R) & := \{ z \in Q_1 : \vert z \vert = R \} \subset \partial\mathcal{R}(r,R) \\
\mathcal{R}_{\text{top}}(r,R) & := \{ (0,y)  : 0 \leq y \leq R \} \subset \partial\mathcal{R}(r,R) \\
\mathcal{R}_{\text{bot}}(r,R) & := \{ z \in Q_{4} : | z | = r \} \cup \{ (x,0) : r \le x \le R \} \subset \partial\mathcal{R}(r,R).
\end{align*}
Notice that $\mathcal{R}(r,R)$ is a \emph{topological rectangle} (that is, it is a planar set homeomorphic to a rectangle); accordingly, the boundary subsets
$\mathcal{R}_{\textnormal{left}}(r,R)$ and $\mathcal{R}_{\textnormal{right}}(r,R)$ have to be thought of as a pair of opposite sides (the ``vertical'' sides),
as well as $\mathcal{R}_{\text{top}}(r,R)$ and $\mathcal{R}_{\text{bot}}(r,R)$ which can be seen as the ``horizontal'' ones.

We are now in position to state the main result of this section, showing how paths in $\mathcal{R}(r,R)$ joining the opposite sides
$\mathcal{R}_{\textnormal{left}}(r,R)$ and $\mathcal{R}_{\textnormal{right}}(r,R)$ are transformed through the map $\varphi^\mu_{\sigma,\omega}$.
Incidentally, by a path in $\mathcal{R}(r,R)$ we simply mean a continuous function from a compact interval
(say, $[0,1]$ for simplicity) into $\mathcal{R}(r,R)$.

\begin{proposition}\label{lemma-sap}
Let $b: [\sigma,\omega] \to \mathbb{R}$ be a continuous function satisfying $(b_*)$ and let
$g: \mathbb{R}^+ \to \mathbb{R}$ be a locally Lipschitz continuous function satisfying $(g_*)$, $(g_0)$ and $(g_\infty)$.
Then, there exist $0 < r_{\sigma,\tau} < R_{\sigma,\tau}$ such that for any $r \in \,]0,r_{\sigma,\tau}]$ and $R \geq R_{\sigma,\tau}$
the following holds true: there exists $\mu^* = \mu^*([\sigma,\omega],r,R) > 0$ such that 
for any $\mu > \mu^*$ and for any continuous path $\gamma: [0,1] \to \mathcal{R}(r,R)$ satisfying
\begin{equation}\label{hp-sap}
\gamma(0) \in 
\mathcal{R}_{\textnormal{left}}(r,R), \qquad
\gamma(1) \in 
\mathcal{R}_{\textnormal{right}}(r,R),
\end{equation}
there exist $0 < \xi_1 < \eta_1 < \eta_2 < \xi_2 < 1 $ such that, for $J_1 := [\xi_1,\eta_1]$ and
$J_2 := [\eta_2,\xi_2]$,
$$
\gamma(J_i) \subset \mathcal{D}'_\mu, \qquad \varphi^\mu_{\sigma,\omega}(\gamma(J_i)) \subset \mathcal{R}(r,R),
$$
and
$$
\varphi^\mu_{\sigma,\omega}(\gamma(\xi_i)) \in \mathcal{R}_{\textnormal{left}}(r,R), \qquad 
\varphi^\mu_{\sigma,\omega}(\gamma(\eta_i)) \in \mathcal{R}_{\textnormal{right}}(r,R), 
$$
for $i=1,2$.
\end{proposition}

The above statement has to be interpreted in a  ``Stretching Along Paths'' spirit
(see \cite{MedPirZan09,PapZan02}), asserting that any path in $\mathcal{R}(r,R)$ joining the opposite sides
$\mathcal{R}_{\textnormal{left}}(r,R)$ and $\mathcal{R}_{\textnormal{right}}(r,R)$ possesses two sub-paths
whose images through the map $\varphi^\mu_{\sigma,\omega}$ are still paths in $\mathcal{R}(r,R)$ joining the same
opposite sides.

We also remark that the notation $r_{\sigma,\tau}, R_{\sigma,\tau}$ is chosen to emphasize
that these values depend only (on $g$ and) on the behavior of $b$ on the interval $[\sigma,\tau]$; 
on the contrary, the constant $\mu^*([\sigma,\omega],r,R)$ depend on the behavior of $b$ on the whole $[\sigma,\omega]$. 
Henceforth, not to overload the notation we will remove the superscript $\mu$ in the maps $\varphi_{\sigma,\omega}$, $\varphi_{\tau,\omega}$.
\smallbreak
The rest of this section is devoted to the proof of Proposition \ref{lemma-sap}.
As a first step, we state two preliminary lemmas, yielding the values $r_{\sigma,\tau}$ and
$R_{\sigma,\tau}$. They are essentially well-known in the literature; however, we sketch the proof for
the sake of completeness.

\begin{lemma} \label{sol-grandi}
There exists $R_{\sigma,\tau} > 0$ such that 
$$
z \in Q_1, \; \vert z \vert \geq R_{\sigma,\tau} \quad \Longrightarrow \quad \varphi_{\sigma,\tau}(z) \in Q_3.
$$ 
\end{lemma}

\begin{proof}
Let us choose $[\sigma',\tau'] \subset [\sigma,\tau]$ in such a way that
$b(t) \geq \underline{b} > 0$ for $t \in [\sigma',\tau']$ and fix $M > 0$ such that
\begin{equation}\label{e1}
\frac{\sqrt{M}}{2}(\tau' - \sigma') > \pi.
\end{equation}
From the superlinearity assumption $(g_\infty)$ we infer the existence of $C > 0$ such that
\begin{equation}\label{e2}
\underline{b} \, g_0(x) x \geq M x^2 - C, \quad \mbox{ for any } x \geq 0;
\end{equation}
then, as a consequence of the global continuability, we find $R > 0$ such that
\begin{equation}\label{e3}
M v(t)^2 + v'(t)^2 \geq 2C, \quad \mbox{ for any } t \in [\sigma,\tau],
\end{equation}
for any solution of \eqref{eq-stre} satisfying $v(\sigma)^2 + v'(\sigma)^2 \geq R^2$.

Let now $z \in Q_1$ with $\vert z \vert \geq R$ and consider the solution $v(\cdot) = v(\cdot;\sigma,z)$
of \eqref{eq-stre} satisfying the initial condition $(v(\sigma),v'(\sigma)) = z$. Assume by contradiction that
$(v(\tau),v'(\tau)) \notin Q_3$; then by the definition of $g_0$ we easily obtain that
$(v(t),v'(t)) \in Q_1 \cup Q_4$ for any $t \in [\sigma,\tau]$. In particular,
$v(t) \geq 0$ for any $t \in [\sigma,\tau]$ and, passing to modified polar coordinates
$$
\sqrt{M} v(t) = \rho(t) \sin\theta(t) \, \qquad v'(t) = \rho(t)\cos\theta(t),
$$ 
a standard computation based on \eqref{e1}, \eqref{e2}, \eqref{e3} gives
$$
\theta(\tau) - \theta(\sigma) = \sqrt{M} \int_{\sigma}^{\tau} \frac{v'(t)^2 + b(t)g_0(v(t))v(t)}{M v(t)^2 + v'(t)^2}\,dt > \pi,
$$
a contradiction. The lemma is thus proved, for $R_{\sigma,\tau}:= R$.
\end{proof}

We now consider the half-line
$$
\mathcal{L} := \{(x,y) \in \mathbb{R}^2 \, :  x \geq 0, \, y = - x \}
$$ 
and the cone
$$
\mathcal{C} := \{(x,y) \in \mathbb{R}^2 \, :  x \geq 0, \, y \geq - x  \}.
$$
\begin{lemma} \label{sol-piccole}
There exists $r_{\sigma,\tau} \in \,]0,R_{\sigma,\tau}[$ such that
$$
z \in Q_1, \; \vert z \vert \leq r_{\sigma,\tau} \quad \Longrightarrow \quad \varphi_{\sigma,\tau}(z) \in \mathcal{C}.
$$ 

\end{lemma}

\begin{proof} 
Let $\bar b=\max_{[\sigma,\tau]} b$ and fix $\epsilon >0$ such that
\begin{equation} \label{aa1}
\sqrt{{\bar b} \epsilon }\  (\tau -\sigma)<\dfrac{\pi}{4}
\end{equation}
and
\begin{equation} \label{aa1bis}
{\bar b} \epsilon <1.
\end{equation}
From assumption $(g_0)$ we deduce that there exists $r'>0$ such that
$$
|x|\leq r'\quad \Rightarrow \quad |g_0(x) x| \leq \epsilon x^2;
$$
then, a continuous dependence argument ensures the existence of $r > 0$ such that
$$
|z|\leq r \quad \Rightarrow \quad |v(t)|\leq r',\quad \forall \ t\in [\sigma,\tau],
$$
where $v(\cdot) = v(\cdot;\sigma,z)$ is the solution of \eqref{eq-stre} satisfying the initial condition $(v(\sigma),v'(\sigma)) = z$.
Hence, for such a solution we have
\begin{equation}\label{aa4}
|g_0(v(t)) v(t)| \leq \epsilon v(t)^2,\quad \forall \ t\in [\sigma,\tau].
\end{equation}

Let us now introduce modified polar coordinates
$$
\sqrt{{\bar b} \epsilon}\, v(t) = \rho_\epsilon(t) \sin \theta_\epsilon(t) \, \qquad v'(t) = \rho_\epsilon(t)\cos\theta_\epsilon(t).
$$ 
A standard computation based on \eqref{aa4} gives
$$
\theta_\epsilon(\tau) - \theta_\epsilon(\sigma) = \sqrt{{\bar b}\epsilon} \int_{\sigma}^{\tau} \frac{v'(t)^2 + b(t)g_0(v(t))v(t)}{{\bar b} \epsilon v(t)^2 + v'(t)^2}\,dt \leq \sqrt{{\bar b}\epsilon}\, (\tau-\sigma);
$$
hence, recalling \eqref{aa1} and the fact that $z\in Q_1$, we obtain
\[
\theta_\epsilon (\tau)\leq \dfrac{\pi}{2}+\dfrac{\pi}{4}.
\]
From this relation together with condition \eqref{aa1bis} we deduce that 
\[
\theta (\tau)\leq \dfrac{\pi}{2}+\dfrac{\pi}{4},
\] 
where $\theta$ is the usual angular coordinate (measured from the $y$-axis), showing that $(v(\tau),v'(\tau))\in \mathcal{C}$.
The lemma is thus proved, for $r_{\sigma,\tau}:= r$.
\end{proof}

\begin{remark}\label{sol-piccole-rad}
For further convenience, we observe that a version of Lemma \ref{sol-piccole} can be proved, when $[\sigma,\tau] = [0,\tau]$, also for the singular ODE
\begin{equation}\label{rad-v}
v'' + \frac{\alpha}{t} v' + b^+(t)g_0(v) = 0, \qquad t \in [0,\tau],
\end{equation}
where $\alpha > 0$. More precisely, the following holds true: \emph{there exists $d_0 > 0$, such that, for any $d \in \,]0,d_0]$, the solution 
$v_d$ of equation \eqref{rad-v} with $(v_d(0),v_d'(0)) = (d,0)$ satisfies $(v_d(\tau),v_d'(\tau)) \in \mathcal{C} \cap Q_4$}.
The proof is essentially the same; here we can estimate the angular coordinate via the formula
$$
\theta_\epsilon(t) - \frac{\pi}{2} = \theta_\epsilon(t) - \theta_\epsilon(0) = \sqrt{{\bar b}\epsilon} \int_{0}^{t} \frac{v'(s)^2 + b(s)g_0(v(s))v(s) + \frac{\alpha}{s}v(s)v'(s)}{{\bar b} \epsilon v(s)^2 + v'(s)^2}\,ds 
$$
and we have to use the fact that $v(s)v'(s) \leq 0$ as long as $(v(s),v'(s)) \in Q_4$.
\end{remark}

We are now in a position to give the proof of Proposition \ref{lemma-sap}.

\begin{proof}[Proof of Proposition \ref{lemma-sap}] 
Let us consider the numbers $R_{\sigma,\tau}$ and $r_{\sigma,\tau}$ given by Lemma \ref{sol-grandi} and Lemma \ref{sol-piccole}, respectively, and fix $r<r_{\sigma,\tau}$ and $R>R_{\sigma,\tau}$. From a compactness argument, together with the uniqueness of the trivial solution of \eqref{eq-stre}, we deduce that there exist $r^*<r$ and $R^*>R$ such that
\begin{equation} \label{bb1}
\vert \varphi_{\sigma,\tau}(z) \vert \geq r_*, \quad \mbox{ for any } \vert z \vert \geq r,
\end{equation}
and
\begin{equation} \label{bb2}
\vert \varphi_{\sigma,\tau}(z) \vert \leq R_*, \quad \mbox{ for any } \vert z \vert \leq R.
\end{equation}
Let us now define
\[
g_*(x)=
\begin{cases}
g_0(x), & \mbox{ if } x \leq R^*, \\
g_0(R^*), & \mbox{ if } x > R^*,
\end{cases}
\]
and consider the modified equation
\begin{equation}\label{eq-mod}
v'' + b_\mu(t)g_*(v) = 0
\end{equation}
together with the flow map
$$
\varphi^*_{\tau,\omega}: \mathbb{R}^2 \to \mathbb{R}^2 \qquad z \mapsto (v(\omega;\tau,z),v'(\omega;\tau,z)),
$$
where $v(\cdot;\tau,z)$ is the solution of \eqref{eq-mod} satisfying the initial condition 
$$(v(\tau;\tau,z),v'(\tau;\tau,z)) = z.
$$
We also set
\[
\varphi^*_{\sigma,\tau}=\varphi_{\sigma,\tau},\qquad \varphi^*_{\sigma,\omega}=\varphi^*_{\tau,\omega}\circ \varphi^*_{\sigma,\tau}.
\]

We claim that the following hold true.

\smallskip
\noindent
{\sc Claim 1.} There exists $\mu_1^*>0$ such that for every $\mu > \mu_1^*$ we have
$$
w\in Q_4, \vert w \vert > r/2\quad \Rightarrow \quad |(\varphi^*_{\tau,\omega})^{-1}(w)| > R^*.
$$
\smallbreak
To prove this, we first observe that, if $\delta \in \,]0,r/2]$ and $v$ is a solution of \eqref{eq-mod} satisfying
$v(\omega) \geq \delta$ and $v'(\omega) \leq 0$, then $v(t) \geq \delta$ for any $t \in [\tau,\omega]$ and
$g_*(v(t)) \geq g_*^\delta := \inf_{v \geq \delta} g_*(v) > 0$. By integrating the differential equation we then obtain
\begin{equation}\label{tecnica}
v'(\tau) \leq - \mu g_*^\delta \int_{\tau}^\omega b^{-}(s)\,ds.  
\end{equation}
We now write $w = (w_1,w_2)$ and we distinguish two cases.
If $w_1 \geq r/4$, from \eqref{tecnica} we immediately see that the thesis of the claim is true for
$$
\mu > \frac{R^*}{g_*^{r/4} \int_{\tau}^\omega b^{-}(s)\,ds}.
$$
On the other hand, if $w_1 \in [0,r/4]$ we first use a simple convexity arguments
to prove that
$$
v(\tau') \geq \delta' := \frac{\sqrt{3}}{4} r (\omega-\tau'),
$$ 
where $\tau' = (\omega+\tau)/2$.
With analogous computations as before, we can thus see that the thesis of the claim is true for
$$
\mu > \frac{R^*}{g_*^{\delta'} \int_{\tau}^{\tau'} b^{-}(s)\,ds}.
$$
\smallbreak
\noindent
{\sc Claim 2.} There exists $\mu_2^*>0$ such that for every $\mu > \mu_2^*$ we have
$$
z\in {\cal L},\ r^*\leq |z|\leq R^*\quad \Rightarrow \quad \pi_i\left(\varphi^*_{\tau,\omega}(z)\right)\geq R,\ i=1, 2,
$$
where $\pi_i(\alpha_1,\alpha_2)=\alpha_i$, $(\alpha_1,\alpha_2)\in \mathbb{R}^2$, $i=1, 2$.
\smallbreak
This can be proved exactly as in \cite[Lemma 3.5]{BosZan12} (up to choosing, without loss of generality, $\tau \in [\sigma,\omega]$
in such a way that $b(t) < 0$ in a small neighborhood to the right of $\tau$).
\smallbreak
\noindent
{\sc Final part.}
Let us define $\mu^* = \max(\mu_1^*,\mu_2^*)$ and fix $\mu > \mu^*$; let also
$\gamma: [0,1] \to \mathcal{R}(r,R)$ be a continuous function such that
$$
H_0 := \gamma(0) \in 
\mathcal{R}_{\textnormal{left}}(r,R), \qquad
H_1 := \gamma(1) \in 
\mathcal{R}_{\textnormal{right}}(r,R),
$$
Preliminarily, we observe that, by \eqref{bb2} together with elementary considerations about the sign of the vector field,
\begin{equation}\label{prel}
\vert \varphi_{\sigma,\tau}(\gamma(s)) \vert \leq R^*,
\end{equation}
and
$$
\varphi_{\sigma,\tau}(\gamma(s)) \in Q_1 \cup Q_3 \cup Q_4
$$ 
for any $s \in [0,1]$. Moreover, 
$K_0:= \varphi_{\sigma,\tau}(H_0) \in Q_3$ and, by Lemma \ref{sol-grandi}
$K_1 := \varphi_{\sigma,\tau}(H_1) \in Q_3$, as well. Now, let $s_2 \in \,]0,1[$ be the greatest value such that
$H_2 := \gamma(s_2) \in Q_1$ and $\vert H_2 \vert = r$; in view of Lemma \ref{sol-piccole} and \eqref{bb1},
$K_2 := \varphi_{\sigma,\tau}(H_2) \in \mathcal{C}$ and $r^* \leq \vert K_2 \vert \leq R_*$.
We then easily obtain the existence of $s_3 \in [s_2,1[$ such that $K_3 := \varphi_{\sigma,\tau}(\gamma(s_3)) \in \mathcal{L}$ and
$\vert K_3 \vert \geq r_*$ (see \eqref{bb1}).

Let us now focus on the interval $[\tau,\omega]$.
Again by the sign of the vector field, we see that 
$$
\varphi_{[\sigma,\omega]}^*(\gamma(s)) = \varphi_{\tau,\omega}^*(\varphi_{\sigma,\tau}(\gamma(s))) \in Q_1 \cup Q_3 \cup Q_4, \quad \mbox{ for any } s \in [0,1],
$$
and $\varphi_{\tau,\omega}^*(K_0) \in Q_3$, $\varphi_{\tau,\omega}^*(K_1) \in Q_3$; moreover, in view of Claim 2, it holds that $\vert \varphi_{\tau,\omega}^*(K_3) \vert \geq R$.
Accordingly, we can define
$$
\xi_1 := \sup \{ s \in [0,s_3] \, : \, \pi_1(\varphi_{[\sigma,\omega]}^*(\gamma(s))) = 0 \}
$$
and
$$
\eta_1 := \inf \{ s \in [\xi_1,1] \, : \, \vert \pi_1(\varphi_{[\sigma,\omega]}^*(\gamma(s))) \vert = R \}.
$$
Again by the sign of the vector field, $\varphi_{\sigma,\omega}^*(\gamma(s)) \in Q_1 \cup Q_4$ for any $s \in [\xi_1,\eta_1]$; moreover,
$\varphi_{\sigma,\omega}^*(\gamma(\xi_1))\notin Q_1$. Finally, \eqref{prel} together with Claim 1 imply that
$$
\varphi_{\sigma,\omega}^*(\gamma(s)) \in Q_4 \quad \Longrightarrow \quad \vert \varphi_{\sigma,\omega}^*(\gamma(s)) \vert \leq r/2.
$$
Hence, 
for $J_1 := [\xi_1,\eta_1]$ we have that $\varphi_{\sigma,\omega}^*(\gamma(J_1)) \subset \mathcal{R}(r,R)$
and 
$$
\varphi_{\sigma,\omega}^*(\gamma(\xi_1)) \in \mathcal{R}_{\textnormal{left}}(r,R), \qquad  
\varphi_{\sigma,\omega}^*(\gamma(\eta_1)) \in \mathcal{R}_{\textnormal{right}}(r,R).
$$

Now, let $v(\cdot)$ be the solution of \eqref{eq-mod} satisfying $(v(\tau),v'(\tau)) = \varphi_{\sigma,\tau}(\gamma(s))$ for $s \in [\xi_1,\eta_1]$. 
We see that $v(t) > 0$ for $t \in [\tau,\omega]$, hence $v$ is convex in $[\tau, \omega]$.
Since $v(\tau) \leq R^*$ and $v(\omega) \leq R$, 
we obtain $v(t) \leq R^*$ for any $t \in [\tau,\omega]$; as a consequence, $v$ solves \eqref{eq-stre} as well. This means that
$\varphi_{\sigma,\tau}(\gamma(s)) \in \mathcal{D}_\mu$ for $s \in [\xi_1,\eta_1]$ and $\varphi_{\tau,\omega}^*(\gamma(s)) = \varphi_{\tau,\omega}(\gamma(s))$,
concluding the proof of the existence of a first sub-path.

The existence of the interval $J_2$ follows from a similar argument, using the sub-path joining $K_3$ and $K_1$.
\end{proof}

\begin{figure}[!h]
\includegraphics[scale=0.6]{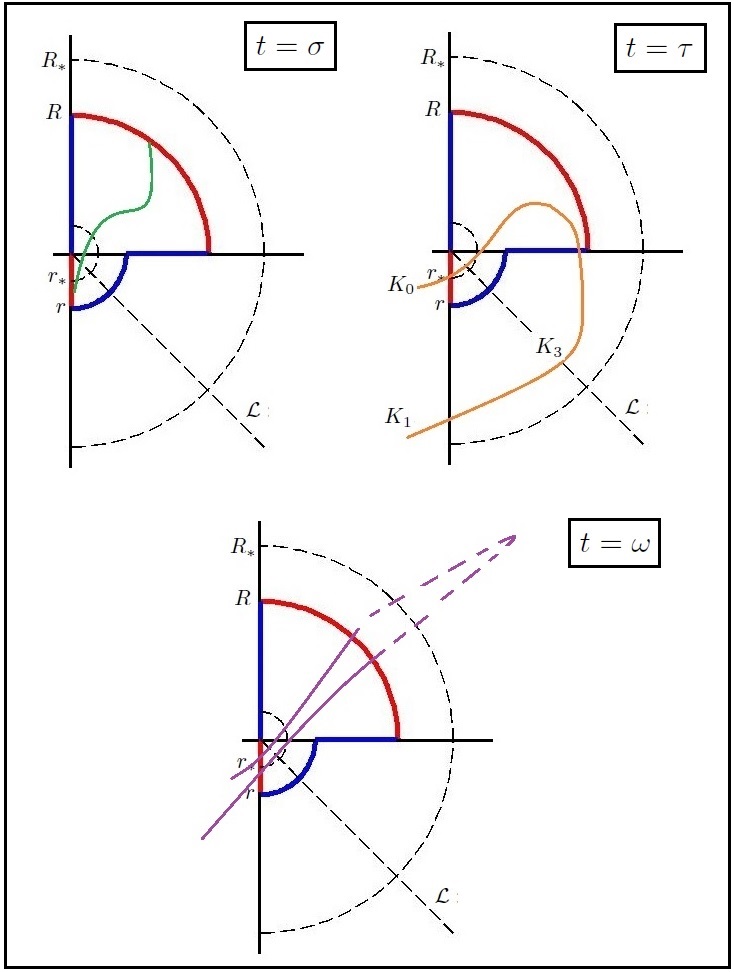}
\caption{\small{A graphical explanation of the proof of Proposition \ref{lemma-sap}. In the first figure, the ``initial'' path $\gamma$ is plotted (in green color), joining the opposite sides $\mathcal{R}_{\textnormal{left}}(r,R), \mathcal{R}_{\textnormal{right}}(r,R)$ (in red color) of the topological rectangle $\mathcal{R}(r,R)$. In the second figure, the path $\varphi_{\sigma,\tau}(\gamma)$ is plotted (in purple color) together with its relevant points $K_0, K_1$ and $K_3$. Finally, in the third figure, the ``final'' path $\varphi^{\mu}_{\sigma,\omega}(\gamma)$ is plotted (in purple color). Notice that, in general, the map $\varphi^{\mu}_{\sigma,\omega}$ is not defined on the whole image of $\gamma$ (in the figure, this fact is expressed by the dashed line); however, two sub-paths joining the opposite sides $\mathcal{R}_{\textnormal{left}}(r,R), \mathcal{R}_{\textnormal{right}}(r,R)$ can be found, in accordance with Proposition \ref{lemma-sap}.}}
\label{fig2}
\end{figure}

\begin{remark}\label{rem-fz}
We observe that Lemma \ref{lemma-sap} can be applied in an iterative way when the weight function 
$b$ changes sign a finite number of times on a compact interval.
As a consequence, by selecting as initial path a half-line
in the first quadrant and following its evolution through the flow map,
one can easily prove the existence of multiple positive solutions
to Sturm-Liouville boundary value problems (e.g., the Dirichlet and the Neumann one) 
associated with equations like $v'' + b_\mu(t) g(v) = 0$, when $\mu > 0$ is large enough.
In this way, it is possible to recover results first obtained by Gaudenzi, Habets and Zanolin \cite{GauHabZan03,GauHabZan04} 
(via a related shooting approach, but for Dirichlet boundary conditions only, and under the simplifying assumption that $b$ has two/three intervals of positivity)
and more recently by Feltrin and Zanolin \cite{FelZan15,FelZanpp} (in greater generality, but using a topological degree approach).
\end{remark}

\subsection{A continuum of blowing-up solutions}\label{sez-2.1}

\noindent
We consider here a continuous weight function $ b : [ \omega, \sigma ] \to \mathbb{R} $ such that $ b(t) \le 0 $
for all $ t \in [ \omega, \sigma ] $.
Therefore, the equation we are considering in this section becomes
\begin{equation}\label{eq:pesoneg}
v'' - \mu b^{-}(t) g(v) = 0, \qquad t \in [ \omega, \sigma ].
\end{equation}
As it is well known, conditions of superlinear growth for $ g $ at infinity like $ ( g^{*}_{\infty} ) $ imply that there are solutions of \eqref{eq:pesoneg}
that blow up in $ [ \omega, \sigma ] $, provided that $ b $ is not trivial (see also Lemma~\ref{lem:scoppiano} below).
More precisely, in \cite{MawPapZan03} it is shown that there are continua of solutions that blow up at a given time.
We recall here that result and afterwards we will give a more precise localization of those continua for large values of $ \mu $.
As before, $ \pi_{1}, \pi_{2} : \mathbb{R}^{2} \to \mathbb{R} $ stand for the orthogonal projections on the $ x $- and $ y $-axes respectively.
\begin{lemma}\label{lem:mapaza}
(\cite[Lemma 1]{MawPapZan03})
Assume $ ( g_{*} )$ and $ ( g_{\infty}^{*} ) $ and suppose that
\[
\sigma \in \overline{\{ t\in ( \omega, \sigma ) : b(t) < 0 \}}.
\]
Then, there is an unbounded continuum $ \Gamma_{\mu} \subset \left[ 0, +\infty \right) \times \mathbb{R} $, with
$ \pi_1(\Gamma_{\mu}) = \left[ 0, +\infty \right)$, such that each solution of \eqref{eq:pesoneg} with
$ ( v(\omega), v'(\omega) ) \in \Gamma_{\mu} $ satisfies $ v(t) > 0 $ for all
$ t \in ( \omega, \sigma ) $ and $ v(t) \to +\infty $ as $ t \to\sigma^{-} $.
Moreover, the localization of $ \Gamma_{\mu} $ in the phase-plane can be described as follows:
there is $ \delta > 0 $ and
\begin{itemize}
\item[(i)]
there is $ \varepsilon > 0 $ such that
$ \pi_2( \Gamma_{\mu} \cap \left[ 0, \varepsilon \right) \times \mathbb{R} ) \subset ( \delta, +\infty ) $;
\item[(ii)]
there is $ K > 0 $ such that $ \pi_2( \Gamma_{\mu} \cap ( K, +\infty ) \times \mathbb{R} ) \subset ( -\infty, -\delta) $.
\end{itemize}
\end{lemma}
We will also show that the part of the continuum $ \Gamma_{\mu} $ lying in the first quadrant is as close as desired to the origin if $ \mu $ is
large enough.
The next lemma provides the needed lower bound for $ \mu $.
\begin{lemma}\label{lem:scoppiano}
Assume that $ g $ satisfies $ ( g_{*} ) $ and $ ( g^{*}_{\infty} ) $ and that there are $ \underline{b} > 0 $ and $ \omega_{1}, \omega_{2} $ such that
$ \omega < \omega_{1} < \omega_{2} \le \sigma $ and $ b(t) \le - \underline{b} $ for $ t \in [ \omega_{1}, \omega_{2} ] $.
Then, for each $ r > 0 $ there exists $ \hat{\mu} = \hat{\mu}( \omega_{1} - \omega, \omega_{2} - \omega_{1}, \underline{b}, r ) > 0 $ such that any solution
$ v $ of \eqref{eq:pesoneg} with $ \left| ( v(\omega), v'(\omega) ) \right| > r $ and $ ( v(\omega), v'(\omega) ) \in Q_{1} $ blows up at
$ t^{*} \in ( \omega, \omega_{2} ) $ if $ \mu > \hat{\mu} $.
\end{lemma}
\begin{proof}
Let $ v $ be a solution of \eqref{eq:pesoneg} with $ \left| ( v(\omega), v'(\omega) ) \right| > r $ and $ ( v(\omega), v'(\omega) ) \in Q_{1} $
and let
\[
t^{*} = \sup \{ t \in ( \omega, \omega_{2}) : v \text{ is continuable on } [ \omega, t ] \}.
\]
The lemma is proved if we show that there exists $ \hat{\mu} $ such that $ t^{*} < \omega_{2} $ whenever $ \mu \ge \hat{\mu} $.
Now, if $ t^{*} \le \omega_{1} $ then there is nothing to prove, therefore, without loss of generality we can assume that
$ t^{*} > \omega_{1} $.
We remark that the sign conditions on $ b $ and $ g $ imply that $ v $ and $ v' $ are both non-negative and increasing on $ \left[ \omega, t^{*} \right) $.

We fix the positive number
\[
\delta = \dfrac{ \omega_{1} - \omega }{ \sqrt{ 1 + ( \omega_{1} - \omega )^{2} } } < 1
\]
and observe that if $ v(\omega) \ge \delta r $ then $ v( t ) \ge \delta r $ for all $ t \in ( \omega, t^{*} ) $ and, in particular, for
$ t = \omega_{1} $.
On the other hand, if $  v(\omega) < \delta r $, then
\[
v'(t) \ge v'(\omega) \ge \sqrt{ r^{2} - v(\omega)^{2} } \ge r \sqrt{ 1 - \delta^{2} }, \quad \mbox{ for every } t \in [ \omega, t^{*}),
\]
and we again obtain
\[
v\left( \omega_{1} \right) \ge r ( \omega_{1} - \omega ) \sqrt{ 1 - \delta^{2} } = \delta r
\]
as a consequence.

We set
\[
E(t) = \frac{1}{2} v'(t)^{2} - \mu \underline{b} G( v(t) )
\]
and compute
\[
E'(t) = \mu v'(t) g( v(t) ) ( b^{-}(t) - \underline{b} ) \ge 0, \quad \mbox{ for every } t \in [ \omega, t^{*} ),
\]
which implies that
\[
E(t) \ge E\left( \omega_{1} \right) \ge -\mu \underline{b} G\left( v\left( \omega_{1} \right) \right), \quad \mbox{ for every } t \in [\omega_{1}, t^{*}),
\]
and, hence,
\[
v'(t) \ge \sqrt{ 2 \mu \underline{b} } \sqrt{ G( v(t) ) - G\left( v\left( \omega_{1} \right) \right) }, \quad \mbox{ for every } t \in [\omega_{1}, t^{*}).
\]
Therefore we obtain the following estimate
\begin{align*}
t^{*} - \omega_{1} & \le \frac{1}{ \sqrt{ 2 \mu \underline{b} } } \int_{ v\left( \omega_{1} \right) }^{ v(t^{*}) }
\dfrac{ d \xi }{ \sqrt{ G(\xi) -  G\left( v\left( \omega_{1} \right) \right) } } \\
& \le \frac{1}{ \sqrt{ 2 \mu \underline{b} } } \sup_{ u \ge \delta r }
\int_{ u }^{ +\infty } \dfrac{ d \xi }{ \sqrt{ G(\xi) -  G\left( u \right) } }. 
\end{align*}
We have that $ t^{*} < \omega_2 $ if we choose
\[
\mu > \hat{\mu} := \frac{1}{ 2 \underline{b} ( \omega_{2} - \omega_{1} )^{2} }
\left[ \sup_{ u \ge \delta r } \int_{ u }^{ +\infty } \dfrac{ d \xi }{ \sqrt{ G(\xi) -  G\left( u \right) } } \right]^{2}
\]
which is finite by assumption $ ( g^{*}_{\infty} ) $.
\end{proof}
Since the solutions starting from $ \Gamma_{\mu} $ at $ t = \omega $ blow up exactly as $ t \to \sigma^{-},$
we immediately obtain the following statement.
\begin{proposition}\label{pro:continuo}
Assume that $ g $ and $ b $ are as in Lemmas~\ref{lem:mapaza} and \ref{lem:scoppiano}.
For each $ r > 0 $ let $ \hat{\mu} = \hat{\mu}( \omega_{1} - \omega, \omega_{2} - \omega_{1}, \underline{b}, r ) > 0 $ be the number given by
Lemma~\ref{lem:scoppiano}.
Then, if $ \mu > \hat{\mu} $, the continuum $ \Gamma_{\mu} $ given by Lemma~\ref{lem:mapaza} satisfies
$ \Gamma_{\mu} \cap Q_{1} \subset \overline{ B( 0, r ) } $.
\end{proposition}
In a similar way we can obtain a backward version of the preceding results which we summarize hereafter without proof.
\begin{proposition}\label{pro:scoppianoasx}
Assume $ ( g_{*} )$ and $ ( g_{\infty}^{*} ) $ and suppose that
\[
\omega \in \overline{\{ t\in ( \omega, \sigma ) : b(t) < 0 \}}.
\]
and that there are $ \omega_{1}, \omega_{2} $ such that $ \omega \le \omega_{1} < \omega_{2} < \sigma $ and
$ b(t) \le - \underline{b} $ for $ t \in [ \omega_{1}, \omega_{2} ] $.
Then, there is an unbounded continuum $ \Gamma_{\mu} \subset \left[ 0, +\infty \right) \times \mathbb{R} $, with
$ \pi_1(\Gamma_{\mu}) = \left[ 0, +\infty \right)$, such that each solution of \eqref{eq:pesoneg} with
$ ( v(\sigma), v'(\sigma) ) \in \Gamma_{\mu} $ satisfies $ v(t) > 0 $ for all
$ t \in ( \omega, \sigma ) $ and $ v(t) \to +\infty $ as $ t \to \omega^{+} $.
The localization of $ \Gamma_{\mu} $ in the phase-plane can be described as follows:
there is $ \delta > 0 $ and
\begin{itemize}
\item[(i)]
there is $ \varepsilon > 0 $ such that
$ \pi_2( \Gamma_{\mu} \cap \left[ 0, \varepsilon \right) \times \mathbb{R} ) \subset ( -\infty, -\delta ) $;
\item[(ii)]
there is $ K > 0 $ such that $ \pi_2( \Gamma_{\mu} \cap ( K, +\infty ) \times \mathbb{R} ) \subset ( \delta, +\infty ) $.
\end{itemize}
Moreover, for each $ r > 0 $ there is $ \hat{\mu} = \hat{\mu}( \sigma - \omega_{2}, \omega_{2} - \omega_{1}, \underline{b}, r ) > 0 $ such that
if $ \mu > \hat{\mu} $ then $ \Gamma_{\mu} \cap Q_{4} \subset \overline{ B( 0, r ) } $.
\end{proposition}
\section{The main results}\label{sez-3}
In this section, we prove our main results dealing with positive radial solutions to the problem
\begin{equation}\label{eq-blow}
\left\{
\begin{array}{ll}
\vspace{0.1cm}
\Delta u + a_\mu(\vert x \vert) g(u) = 0, & \; x \in B,
\\
u(x) \to \infty, & \; x \to \partial B,
\end{array}
\right.
\end{equation}
where $B := \{ x \in \mathbb{R}^N \, : \, \vert x \vert < 1\}$ is the unit ball and the weight function $a_\mu$ is defined as
$$
a_\mu(r) := a^+(r) - \mu a^-(r), \qquad \mu > 0,
$$
with $a^+,a^-$ the positive/negative part of a continuous function $a: [0,1] \to \mathbb{R}$ satisfying the following condition:
\begin{itemize}
\item [$(a_*)$] 
\textit{$a(1) < 0$ and there exist points $ \tau_{i}, \sigma_{i} $ such that
\begin{equation*}
0 = \tau_{0} \leq \sigma_{1} < \tau_{1} < \ldots < \sigma_{i} < \tau_{i} < \ldots < \sigma_{m} < \tau_{m} < \sigma_{m+1} = 1
\end{equation*}
and
\begin{align*}
& a(r)\geq 0, \; \text{ on } [\sigma_{i},\tau_{i}], \qquad a(r)\not\equiv0 \; \text{ on } [\sigma_{i},\tau_{i}], \qquad \quad i=1,\ldots,m; \\
& a(r)\leq 0, \; \text{ on } [\tau_{i},\sigma_{i+1}], \quad a(r)\not\equiv0 \; \text{ on } [\tau_{i},\sigma_{i+1}], \qquad i=0,\ldots,m.
\end{align*}
}
\end{itemize}
Incidentally, when referring to radial solutions to \eqref{eq-blow} we always mean
\emph{classical} radial solutions, namely, solutions $u(x) = u(r)$ (with $r = \vert x \vert)$) of the singular ODE problem
\begin{equation}\label{ODE-1}
\left\{
\begin{array}{ll}
\vspace{0.1cm}
\bigl(r^{N-1} u'\bigr)' + r^{N-1} a_\mu(r) g(u) = 0, & \; 0 < r < 1,
\\
u'(0) = 0, \quad \lim_{r \to 1^-}u(r) = +\infty, & 
\end{array}
\right.
\end{equation}
\begin{remark}\label{oss-blo}
Notice that the condition $a \leq 0$ on a left neighborhood of $r =1$ is necessary for the existence of solutions 
to \eqref{ODE-1}. By elementary arguments, it is also easily seen that $\lim_{r \to 1^-} u'(r) = +\infty$.
\end{remark}

Our first main result treats the case when $a$ is non-positive near $r = 0$.

\begin{theorem}\label{th-main1}
Let $a: [0,1] \to \mathbb{R}$ be a continuous function satisfying $(a_*)$ with
$$
0 < \sigma_1.
$$
Let $g: \mathbb{R}^+ \to \mathbb{R}$ be a locally Lipschitz continuous function satisfying $(g_*)$, $(g_0)$,
$(g_\infty)$ and $(g_\infty^*)$. Then, there exists $\mu^* > 0$ such that for any $\mu > \mu^*$ problem \eqref{eq-blow} has at least $2^m$ distinct positive radial solutions.
\end{theorem}

Our second result deals with the case when $a$ is positive near $r = 0$; here, however, we need a further assumption (see Remark \ref{rem-dir}).

\begin{theorem}\label{th-main2}
Let $a: [0,1] \to \mathbb{R}$ be a continuous function satisfying $(a_*)$ with
$$
0 = \sigma_1.
$$
Let $g: \mathbb{R}^+ \to \mathbb{R}$ be a locally Lipschitz continuous function satisfying $(g_*)$, $(g_0)$,
$(g_\infty)$ and $(g_\infty^*)$. Finally, assume that there exists a positive radial solution of the Dirichlet problem
\begin{equation}\label{hp-ast}
\left\{
\begin{array}{ll}
\vspace{0.1cm}
\Delta u + a(\vert x \vert) g(u) = 0, & \; \vert x \vert < \tau_1,
\\
u(x) = 0, & \; \vert x \vert = \tau_1.
\end{array}
\right.
\end{equation} 
Then, there exists $\mu^* > 0$ such that for any $\mu > \mu^*$ problem \eqref{eq-blow} has at least $2^m$ distinct positive radial solutions.
\end{theorem}

\begin{remark}\label{rem-dir}
Some remarks on the assumption on the existence of a positive radial solution to \eqref{hp-ast} are in order. 
Indeed, as it is well known, solutions to the equation in \eqref{hp-ast} satisfy the Pohozaev-type identity 
\begin{equation}\label{po}
\int_{B_{\tau_1}} \Delta A(\vert x \vert) G(u) \,dx - \frac{N-2}{2} \int_{B_{\tau_1}} a(\vert x \vert) u g(u) \, dx = 
\frac{1}{2} \int_{\partial B_{\tau_1}} u_{\nu}^2 \vert x \vert \, dx,
\end{equation}
where $B_{\tau_1}$ is the ball of radius $\tau_1$ and $A'(r) = r a(r)$ for $r \in [0,1]$, showing that the existence of a positive solution to \eqref{hp-ast} cannot be guaranteed for every $a$ and $g$; however, several existence results, depending on the assumptions on $a$ and $g$, can be proved.
In particular, whenever $g$ satisfies the usual Ambrosetti-Rabinowitz condition
\begin{equation} \label{A-R}
G(u)\leq \alpha ug(u),\quad \mbox{ for every } u \gg 0, 
\end{equation}
for some $\alpha \in (0,1/2)$, together with
\[
g(u)=O(u^p), \qquad u\to +\infty,
\]
for $p$ subcritical in the sense of Sobolev embeddings, that is,
\begin{equation}\label{subcrit}
p < \frac{N+2}{N-2},
\end{equation}
then standard variational arguments (see \cite{AmbMal07}) yields the existence of a positive solution to \eqref{hp-ast}.
The sub-criticality assumption \eqref{subcrit} could be even relaxed into   
\begin{equation}\label{subcritl}
p < \frac{N+2+2l}{N-2},
\end{equation}
whenever $a(0)=0$ and $a$ is Holder-continuous of some order $l>0$ in a neighborhood of $r=0$ (see \cite{Ni82};
incidentally, we observe that, in view of \eqref{po}, the bound \eqref{subcritl} is sharp for $g(u) = u^p$ and $a(r) = r^l$).
On the other hand, some complementary results not requiring \eqref{A-R} are also available. For instance, a shooting approach 
on the lines of \cite{Dam04} (see also \cite{CasKur87,GarManZan97}) gives the existence of a positive radial solution to \eqref{hp-ast} whenever $a(0) > 0$ and $g$ satisfies
\begin{equation} \label{hp-crit}
\liminf_{u \to +\infty} \frac{G(\theta u)}{u \hat g(u)} > \dfrac{N-2}{2N},
\end{equation}
for some $\theta \in (0,1)$, where
$$
\hat g(u):= \sup_{0 \leq v \leq u} g(v).
$$
Notice that, in the model case $g(u) = u^p$ for $p > 1$ and $N \geq 3$, \eqref{hp-crit} and \eqref{subcrit} are actually equivalent.
With the same technique, the case $a(0) = 0$ can also be tretated by assuming $a$ is continuously differentiable
in some interval $]0,\varepsilon]$ and $a(r) \sim c r^l$ with $l > 0$; in this situation, \eqref{hp-crit} can be relaxed, giving rise, in the model case $g(u) = u^p$, to the sub-criticality assumption \eqref{subcritl}.
\end{remark}

The rest of the section is devoted to the proof of our main results.

\begin{proof}[Proof of Theorem \ref{th-main1}]
We adopt a shooting-type approach, namely, we aim at showing the existence of
$$
0 < s_1 < \ldots < s_i < \ldots < s_{2^m}
$$
such that the solution $u_s$ of the Cauchy problem
\begin{equation}\label{blowODE2}
\left\{
\begin{array}{ll}
\vspace{0.1cm}
\bigl(r^{N-1} u'\bigr)' + r^{N-1} a_\mu(r) g(u^+) = 0, 
\\
(u(0), u'(0)) = (s,0),  
\end{array}
\right.
\end{equation}
is defined on $[0,1)$ and satisfies $\lim_{t \to 1^-}u_s(t) = +\infty$.
Incidentally, notice that by standard maximum principle arguments solutions obtained in this way are strictly positive, thus giving rise to positive radial solutions to \eqref{eq-blow}.

As a first step, we are going to consider the equation on the interval $[0,\sigma_1]$. Let us set
$$
s_\mu := \sup \left\{ \bar{d} \, : \, u_s \textnormal{ is defined (at least) on $[0,\sigma_1]$ for any $s \in [0,\bar{d}]$} \right\}.
$$
Notice that $s_\mu$ is well defined and strictly positive by continuous dependence arguments ($u \equiv 0$ is a solution of the equation); moreover the map
$$
[0,s_\mu[ \, \times \{0\} \ni (s,0) \mapsto \psi_{[0,\sigma_1]}(s,0) := (u_s(\sigma_1),u_s'(\sigma_1))
$$ 
is injective, satisfies $\psi_{[0,\sigma_1]} (0,0) = 0$ and
\begin{equation}\label{p1}
\psi_{[0,\sigma_1]} \left( [0,s_\mu[ \, \times \{0\} \right) \subset Q_1, \qquad \lim_{s \to s_\mu^-} \vert \psi_{[0,\sigma_1]}(s,0) \vert = +\infty. 
\end{equation}
Indeed, an elementary argument shows that $u_s' > 0$ as long as $u_s > 0$, thus proving the first part of \eqref{p1}; the second one, instead, follows from standard compactness arguments. 

We want now to follow the evolution of the path $[0,s_\mu[ \, \ni s \mapsto \psi_{[0,\sigma_1]}(s,0)$ through the flow
map on the interval $[\sigma_1,1]$. It is convenient, however, to change coordinates by setting
$$
t = h(r) := \int_{\sigma_1}^r \xi^{1-N}\,d\xi, \qquad v(t) := u(h^{-1}(t));
$$  
In this way, the equation in \eqref{blowODE2} is transformed into
\begin{equation}\label{eq:nonsingolare}
v'' + b_\mu(t) g_0(v) = 0, \qquad t \in [0,h(1)],
\end{equation}
where
\[
b_\mu(t) = b^{+}(t) - \mu b^{-}(t) \quad \text{and} \quad b(t) =  (h^{-1}(t))^{2N-2}a(h^{-1}(t)).
\]
Let us observe that $v'(t) = u'(r(t)) (h^{-1})'(t)$; hence, we are naturally led to consider the Cauchy problem
\begin{equation}\label{Cau2}
\left\{
\begin{array}{ll}
\vspace{0.1cm}
v'' + b_\mu(t)g(v^+) = 0, 
\\
(v(0), v'(0)) = \gamma(s),  
\end{array}
\right.
\end{equation}
where
\begin{equation}\label{defgamma}
\gamma(s) := T(\psi_{[0,\sigma_1]}(s,0)), \qquad s \in [0,s_\mu[\,,  
\end{equation}
and $ T(x,y) := (x,(h^{-1})'( h(\sigma_{1})) y ) = ( x, \sigma_{1}^{N-1} y ) $ for any $(x,y) \in \mathbb{R}^2$.
To conclude the proof we thus have to show the existence of
$$
0 < s_1 < \ldots < s_i < \ldots < s_{2^m} < s_\mu
$$
such that the solution $v_s$ of \eqref{Cau2} satisfies $\lim_{t \to h(1)^-} v_s(t) = +\infty$.  

To prove this, we are going to take advantage of the results developed in Section \ref{sez-2}.
Precisely, we define
\[
\sigma_i' := h(\sigma_i), \quad i=1,\ldots,m+1, \qquad \tau_i' := h(\tau_i), \quad i=1,\ldots,m, 
\]
and we fix $ \omega_{m}' $ such that
\[
\tau_m' < \omega_m' < \sigma'_{m+1} \quad \text{and}\quad b \not\equiv 0 \text{ on } [ \tau_{m}', \omega_{m}' ] \text{ and on }
[\omega_{m}', \sigma_{m+1}'].
\]
It is easily seen that the assumption $(b_*)$ of Section \ref{sez-2.2} is satisfied on each interval
$[\sigma_i',\sigma_{i+1}']$ (with the choices $ \sigma = \sigma_{i}'$, $ \omega = \sigma_{i+1}'$, $ \tau = \tau_{i}' $) for $i=1,\ldots,m-1$
as well as on the interval $[\sigma_m',\omega_m']$ (with the choices $ \sigma = \sigma_{m}'$, $ \omega = \omega_{m}'$, $ \tau = \tau_{m}' $).
Moreover, there exist $ \omega_{m}'',$ $ \omega_{m}'''$ and $ \underline{b} > 0 $ such that
\[
\omega_{m}' < \omega_{m}'' < \omega_{m}''' < \sigma_{m+1}' \quad \text{and} \quad b(t) \ge -\underline{b} \text{ on } [ \omega_{m}'', \omega_{m}''' ].
\]
Accordingly, we can consider the values
$$
0 < r_{\sigma_i',\tau_i'} < R_{\sigma'_i,\tau_i'}, \qquad i=1,\ldots,m,
$$
given by Proposition \ref{lemma-sap} and we can define the topological rectangle $\mathcal{R}(r,R)$ as in \eqref{def-ret}, where
$$
0 < r < \min_i r_{\sigma_i',\tau_i'} < \max_i R_{\sigma'_i,\tau_i'} < R.
$$
Moreover, we fix
\begin{align*}
\mu  > \max\left\{\vphantom{\max_{1\leq i\leq m-1}} \right. & \max_{1\leq i\leq m-1} \mu^*([\sigma'_i,\sigma_{i+1}'],r,R), \;\mu^*([\sigma'_m,\omega_m'],r,R), \\
& \left. \hat\mu( \omega_{m}'' - \omega_{m}', \omega_{m}''' - \omega_{m}'', \underline{b}, r ) \vphantom{\max_{1\leq i\leq m-1}} \right\},
\end{align*}
where all the values $\mu^*$ in the above expression are given by Proposition \ref{lemma-sap}, while the value $ \hat\mu $ is given
by Proposition \ref{pro:continuo}.

In view of \eqref{p1} and recalling the definition \eqref{defgamma}, we can find $0 < s^* < s_\mu$ such that the path
$\gamma$ satisfies $\gamma([0,s^*]) \subset \mathcal{R}(r,R)$ as well as \eqref{hp-sap}.
Therefore, applying $m$ times Proposition \ref{lemma-sap}
we obtain the existence of $2^m$ pairwise disjoint intervals $[s_k^-,s_k^+]$, for $k=1,\ldots,2^m$, such that
$$
[s_k^-,s_k^+] \ni s \mapsto \varphi_{[\sigma_1',\omega_m']} (\gamma(s)) = \varphi_{[\tau_m',\omega_m']} \circ \varphi_{[\sigma_m',\tau_m']}
\circ \cdots \circ \varphi_{[\sigma_1',\tau_1']}(\gamma(s))  
$$
is a path contained in $\mathcal{R}(r,R)$ and joining the opposite sides
$\mathcal{R}_{\textnormal{left}}(r,R)$ and $\mathcal{R}_{\textnormal{right}}(r,R)$ (we are using here 
the natural notation for the flow map introduced in Section \ref{sez-2.2}).

On the other hand, thanks to Lemma \ref{lem:mapaza}, our choice of $ \mu $ and Proposition \ref{pro:continuo}, there exists an unbounded continuum
$ \Gamma_{\mu} \subset \left[ 0, +\infty \right) \times \mathbb{R} $ such that all solutions of \eqref{eq:nonsingolare} such that
$ ( v( \omega_{m}' ), v'( \omega_{m}' ) ) \in \Gamma_{\mu} $ are positive in $ ( \omega_{m}', h(1) ) $ and satisfy $ v( t ) \to +\infty $
as $ t \to \sigma'^{-}_{m+1} $.
Moreover, the localization properties given by Lemma \ref{lem:mapaza} and Proposition \ref{pro:continuo} ensure that the intersection
$ \Gamma_{\mu} \cap \mathcal{R}(r,R) $ has a connected component $ \Gamma_{\mu}' \subset \Gamma_{\mu} $ such that
\[
\{ 0 \} \times \left( 0, r \right] \supset \Gamma_{\mu}' \cap \{ 0 \} \times \mathbb{R} \neq 
 \emptyset \neq \Gamma_{\mu}' \cap Q_{4} \cap \partial B(0,r).
\]
In other words, there is a subcontinuum $ \Gamma_{\mu}' $ of $ \Gamma_{\mu} $ which lies inside the topological rectangle $ \mathcal{R}(r,R) $ and
joins the two ``horizontal'' sides $ \mathcal{R}_{\text{top}}(r,R) $ and $ \mathcal{R}_{\text{bot}}(r,R) $ of $ \partial \mathcal{R}(r,R) $.
Therefore, we apply \cite[Lemma 3]{MulWil74} and find that each path $ \varphi_{[\sigma_1',\omega_m']} (\gamma([s_k^-,s_k^+])) $ intersects
$ \Gamma_{\mu}' $ in at least one point.
More precisely, this means that for each $ k = 1, \dots, 2^{m} $, there exists $ s_{k} \in ( s_{k}^{-}, s_{k}^{+} ) $ such that
$ \varphi_{[\sigma_1',\omega_m']} (\gamma(s_{k})) \in \Gamma_{\mu} $ and, thus, the solution of \eqref{blowODE2} with $ s = s_{k} $ gives rise to
a solution of \eqref{eq-blow}.
\end{proof}

\begin{proof}[Proof of Theorem \ref{th-main2}]
As in the proof of Theorem \ref{th-main1}, we want to show the existence of
$$
0 < s_1 < \ldots < s_i < \ldots < s_{2^m}
$$
such that the solution $u_s$ of the Cauchy problem \eqref{blowODE2}
is defined on $[0,1)$ and satisfies $\lim_{t \to 1^-}u_s(t) = +\infty$.
Here, however, it is convenient to split the study of the dynamics on the intervals
$[0,\sigma_2]$ and $[\sigma_2,1]$.

As for the dynamics on $[\sigma_2,1]$, we argue exactly as in the proof of Theorem \ref{th-main1}. That is, 
we change variables by setting
$$
t = h(r) := \int_{\sigma_2}^r \xi^{1-N}\,d\xi, \qquad v(t) := u(h^{-1}(t)),
$$  
we define
$$
\sigma_i' := h(\sigma_i), \quad i=2,\ldots,m+1, \qquad \tau_i' := h(\tau_i), \quad i=2,\ldots,m, 
$$
and we fix 
$$
\tau_m' < \omega_m' < \sigma'_{m+1}.
$$
Then, we consider the values
$$
0 < r_{\sigma_i',\tau_i'} < R_{\sigma'_i,\tau_i'}, \qquad i=2,\ldots,m,
$$
given by Proposition \ref{lemma-sap} and we define the topological rectangle $\mathcal{R}(r,R)$ for
$$
0 < r < \min_i r_{\sigma_i',\tau_i'} < \max_i R_{\sigma'_i,\tau_i'} < R.
$$
Finally, we take
$$
\mu > \max\left\{\max_{2\leq i\leq m-1} \mu^*([\sigma'_i,\sigma_{i+1}'],r,R), \mu^*([\sigma'_m,\omega_m'],r,R)\right\},
$$
where all the values $\mu^*$ in the above expression are given again by Proposition \ref{lemma-sap}.

Now, we consider the dynamics on the interval $[0,\sigma_2]$. Defining (on its natural domain) the flow map
$$
\psi_{[0,\sigma_2]}(s,0) = (u_s(\sigma_2),u_s'(\sigma_2)),
$$
our aim is to show that the existence
of 
$$
0 < S^-_1 < S^+_1 < S^-_2 < S^+_2
$$
such that, for $k=1,2$ and $\mu$ large enough, the path  
$$
[S_k^-,S_k^+] \ni s \mapsto \eta_k(s) := T(\psi_{[0,\sigma_2]} (s,0)),    
$$
(here $T(x,y) := (x,(h^{-1})'(\sigma_2)y)$) 
is contained in the rectangle $\mathcal{R}(r,R)$ and joins the opposite sides
$\mathcal{R}_{\textnormal{left}}(r,R)$ and $\mathcal{R}_{\textnormal{right}}(r,R)$.
If this is the case, we can apply $m-1$ times Proposition \ref{lemma-sap} to each 
$\eta_i$ and then conclude the proof as in the one of Theorem \ref{th-main1}, by showing that each of the resulting $2^m$ sub-paths
actually intersects the blow-up continuum.

We thus conclude the proof by constructing the above paths $\eta_i$. Let us define $S^*$ to be the value
$u(0)$ for the positive radial solution of \eqref{hp-ast} and set
$\gamma(s) := (s,0)$ for $s \in [0,S^*]$.
We first consider the dynamics on $[0,\tau_1]$; 
incidentally, we notice that the solution
$u_s$ is defined on the whole $[0,\tau_1]$ for any $s \in [0,S^*]$.
Let $H_0 := \gamma(0) = ( 0, 0 ) $ and $H_1 := \gamma(S^*)$;
then, $K_0 := \psi_{[0,\tau_1]}(H_0) = (0,0)$ and, by assumption \eqref{hp-ast}, $K_1 := \psi_{[0,\tau_1]}(H_1) = (0,Y)$ with $Y < 0$.
Moreover, according to Remark \ref{sol-piccole-rad} we can find $s_3 \in \,]0,S^*[$ such that $K_3 := \psi_{[0,\tau_1]}(\gamma(s_3)) \in \mathcal{L}$.
Now, we pass to the dynamics on $[\tau_1,\sigma_2]$. Here, we argue exactly as in the final part of the proof of Proposition \ref{lemma-sap};
notice that Claim 1 and Claim 2 used therein still hold true, since, using the usual change of variable, we can transform the radial equation 
$(r^{N-1}u')' - \mu r^{N-1} a^-(r)g(u) = 0$ into $v'' - \mu b^-(t) g(v) = 0$ 
(for some $b^-$). The conclusion is then obtained for $\mu$ large enough.
\end{proof}

\begin{remark}
It is worth noticing that a lower bound on $\mu$ can be given whenever $g'(u) > 0$ for any $u > 0$. Indeed, in this case
the existence of positive radial solutions to \eqref{eq-blow} is possible only if
\begin{equation}\label{musharp}
\mu > \mu^{\#} : = \displaystyle{\dfrac{\int_B a^+(\vert x \vert)\,dx}{\int_B a^-(\vert x \vert)\,dx} = 
\frac{\int_0^1 r^{N-1} a^+(r)\,dr}{\int_0^1 r^{N-1}a^-(r)\,dr}}.
\end{equation}
To see this, we write the equation in \eqref{ODE-1} as
$$
r^{N-1}a_\mu(r) = -\frac{\bigl(r^{N-1} u'(r)\bigr)'}{g(u(r))}
$$
and we integrate on $[0,1-\eta]$, with $\eta > 0$ small, so as to obtain
$$
\int_0^{1-\eta} r^{N-1}a_\mu(r)\,dr = - \frac{(1-\eta)^{N-1} u'(1-\eta)}{g(u(1-\eta))} 
-\int_0^{1-\eta} \frac{r^{N-1}u'(r)^2}{g(u(r))^2}g'(u(r))\,dr.
$$ 
Recalling Remark \ref{oss-blo} (implying $u'(1-\eta) > 0$) and passing to the limit $\eta \to 0^+$, we finally find
$$
\int_0^{1} r^{N-1}a_\mu(r)\,dr < 0,
$$
thus yielding \eqref{musharp}. Conditions of this type were first introduced for the Neumann problem by Bandle, Pozio and Tesei \cite{BanPozTes88}; as for 
boundary blow-up solutions to the genuine PDE problem, see also \cite[Thorem 1]{Gar11}.
\end{remark}

\section{Related results}\label{sez-4}

In this final section, we propose some further results which can be easily obtained
using shooting-type arguments  
on the lines of the ones developed throughout the paper. 
\smallbreak
At first, we deal with a one-dimensional blow-up problem, by looking for (positive) solutions blowing-up at the extreme points of a compact interval. 

\begin{theorem}
Let $a: [0,1] \to \mathbb{R}$ be a continuous function such that $a(0) < 0$, $a(1) < 0$ and there exist 
$$
0 = \tau_{0} < \sigma_{1} < \tau_{1} < \ldots < \sigma_{i} < \tau_{i} < \ldots < \sigma_{m} < \tau_{m} < \sigma_{m+1} = 1
$$
such that
\begin{align*}
& a(r)\geq 0, \; \text{ on } [\sigma_{i},\tau_{i}], \qquad a(r)\not\equiv0 \; \text{ on } [\sigma_{i},\tau_{i}], \qquad \quad i=1,\ldots,m; \\
& a(r)\leq 0, \; \text{ on } [\tau_{i},\sigma_{i+1}], \quad a(r)\not\equiv0 \; \text{ on } [\tau_{i},\sigma_{i+1}], \qquad i=0,\ldots,m.
\end{align*}
Let $g: \mathbb{R}^+ \to \mathbb{R}$ be a locally Lipschitz continuous function satisfying $(g_*)$, $(g_0)$,
$(g_\infty)$ and $(g_\infty^*)$. Then, there exists $\mu^* > 0$ such that for any $\mu > \mu^*$ the blow-up problem
\begin{equation}\label{blow-fin}
\left\{
\begin{array}{ll}
\vspace{0.1cm}
u'' + a_\mu(t)g(u) = 0, 
\\
\lim_{t \to 0^+} u(t) = \lim_{t \to 1^-}u(t) = +\infty,  
\end{array}
\right.
\end{equation}
has at least $2^m$ distinct positive solutions.
\end{theorem}

The proof is based again on the results of Section \ref{sez-2}. However, despite the fact that the equation now is simpler, some care is needed
due to the fact that two different continua of blowing-up solutions (at $t = 0^+$ and at $t = 1^-$ respectively) are involved in the argument. 
Let us also mention that multiple \emph{oscillating} blow-up solutions to \eqref{blow-fin} were already found (for any $\mu > 0$) in \cite{MawPapZan03}.

\begin{proof}[Sketch of the proof]
We fix $ \alpha_{0}', \alpha_{0}'', \omega_m, \omega_{m}', \omega_{m}'', $
and $ \underline{a} > 0 $ such that $ 0 \le \alpha_{0}' < \alpha_{0}'' < \sigma_{1} $,
$ \tau_{m} < \omega_{m} < \omega_{m}' < \omega_{m}'' \le 1 $, and
$ a(t) \le -\underline{a} $ on $ [ \alpha_{1}', \alpha_{1}'' ] \cup [ \omega_{m}', \omega_{m}'' ] $.
and we consider the continua $\Gamma_{\mu}^{0}$ and $\Gamma_{\mu}^{1} $ made up by initial conditions
(at the times $ \sigma_{1} $ and $ \omega_{m} $, respectively) of solutions blowing-up as
$ t \to 0^+ $ and $ t \to 1^-$ respectively.
More precisely, we apply Proposition \ref{pro:scoppianoasx} on $ [ 0, \sigma_{1} ] $ to get $ \Gamma_{\mu}^{0} $ and  Lemma \ref{lem:mapaza}
and Proposition \ref{pro:continuo} on $ [ \omega_{m}, 1 ] $ to get $ \Gamma_{\mu}^{1} $.
It can be shown that $ \Gamma_{\mu}^{0} $ crosses any topological rectangle $ \mathcal{R}(r,R) $ joining its ``vertical'' sides $ \mathcal{R}_{\text{left}} $
and $ \mathcal{R}_{\text{right}}$, provided that $ \mu $ is sufficiently large.
We thus have to show that there exist $2^m$ positive solutions to the geometrical Sturm-Liouville problem
\begin{equation}\label{SL}
\begin{cases}
u'' + a_\mu(t)g(u) = 0, & t \in [ \sigma_{1}, \omega_{m} ] \\
(u(\sigma_1),u'(\sigma_1)) \in \Gamma_{\mu}^{0} & \\
(u(\omega_m),u'(\omega_m)) \in \Gamma_{\mu}^{1} &
\end{cases}
\end{equation}
for each $ \mu $ large enough.
This can be proved by using in an iterative way Lemma \ref{lemma-sap} similarly as in the proof of Theorem \ref{th-main1} with a suitable topological
rectangle $ \mathcal{R}(r,E) $: 
the only difference here is that we have to follow the evolution of a general continuum $ \Gamma_{\mu}^{0} $ instead of the evolution
of the image of a continuous curve under the flow generated by the differential equation in \eqref{SL} in the phase plane.
One way to overcome this difficulty is to approximate the portion of the continuum $ \Gamma_{\mu}^{0} \cap \mathcal{R}(r,R) $ by a path
$ \gamma^{\epsilon} : [ 0, 1 ] \to \mathcal{R}(r,R) $ ($ \epsilon > 0 $ arbitrary and small) such that its image lies in an $ \epsilon $-neighborhood of
$ \Gamma_{\mu}^{0} $ and satisfies $ \gamma^{\epsilon}(0) \in \mathcal{R}_{\text{left}} $ and $ \gamma^{\epsilon}(1) \in \mathcal{R}_{\text{right}}$.
The details of this standard approximation procedure can be found for instance in \cite[Section 4, Claim 1]{DamPap->}.
The arguments employed in the proof of Theorem \ref{th-main1} show that, if $ \mu $ is large enough,
there are points $ s_{k}^{\epsilon} $, $ k = 1, \dots, 2^{m} $, with:
\[
0 < s_{1} < s_{2} < \dots < s_{2^{m}} < 1,
\]
such that the solution $ u_{k}^{\epsilon} $ of
\[
\begin{cases}
u'' + a_\mu(t)g(u) = 0, & t \in [ \sigma_{1}, \omega_{m} ] \\
(u(\sigma_1),u'(\sigma_1)) = \gamma^{\epsilon}(s_{k})
\end{cases}
\]
satisfies also $ ( u^{\epsilon}(\omega_m), (u^{\epsilon})'(\omega_m) ) \in \Gamma_{\mu}^{1} $.
Letting $ \epsilon \to 0 $ along a suitable sequence, $ \gamma^{\epsilon}(s_{k} $ converges to $ z_{k} \in \Gamma_{\mu}^{0} \cap \mathcal{R}(r,R) $ and, thus,
$ u_{k}^{\epsilon} $ converges (uniformly on $ [ \sigma_{1}, \omega_{m} ] $) to a solution $ u_{k} $ of \eqref{SL}.
\end{proof}

We now conclude the paper by showing how the stretching-type Proposition \ref{lemma-sap}
can be applied, within a shooting approach, also in different contexts.
More precisely, instead of considering blow-up solutions,  
in our final result we look for positive radial homoclinic
solutions to
\begin{equation}\label{eq-hom}
\Delta u + a_\mu(\vert x \vert) g(u) = 0, \qquad x \in \mathbb{R}^N,
\end{equation}
when $a$ is negative at infinity (see \cite{FranSfecNARWA17, FranSfecARXIV16} for some very recent and related results).
For simplicity, we restrict ourselves to the case of a weight function which is non-positive near $r=0$
(as in Theorem \ref{th-main1}) but clearly a similar conclusion could be obtained also when $a$ is non-negative at the expenses of 
additional sub-criticality assumptions (as in Theorem \ref{th-main2}; see also Remark \ref{rem-dir}). Incidentally, we recall that by a radial homoclinic solution to \eqref{eq-hom}
we mean a solution $u(x) = u(\vert x \vert)$ with $(u(r),u'(r)) \to (0,0)$ for $r \to +\infty$.

\begin{theorem}
Let $a: [0,+\infty[\, \to \mathbb{R}$ be a continuous function such that there exist 
$$
0 = \tau_{0} < \sigma_{1} < \tau_{1} < \ldots < \sigma_{i} < \tau_{i} < \ldots < \sigma_{m} < \tau_{m}
$$
such that
\begin{align*}
& a(r)\geq 0, \; \text{ on } [\sigma_{i},\tau_{i}], \qquad a(r)\not\equiv0 \; \text{ on } [\sigma_{i},\tau_{i}], \qquad \quad i=1,\ldots,m; \\
& a(r)\leq 0, \; \text{ on } [\tau_{i},\sigma_{i+1}], \quad a(r)\not\equiv0 \; \text{ on } [\tau_{i},\sigma_{i+1}], \qquad i=0,\ldots,m - 1; \\
& a(r) < 0, \; \text{ on } \,]\tau_m,+\infty[ \quad \mbox{ with } \lim_{r \to +\infty} \int_{\tau_m}^r a(s) s^{N-1}\,ds = - \infty. 
\end{align*}
Let $g: \mathbb{R}^+ \to \mathbb{R}$ be a locally Lipschitz continuous function satisfying $(g_*)$, $(g_0)$,
$(g_\infty)$ and $(g_\infty^*)$. Then, there exists $\mu^* > 0$ such that for any $\mu > \mu^*$ equation
\eqref{eq-hom} has at least $2^m-1$ distinct positive radial homoclinic solutions.
\end{theorem}

The strategy of the proof is very similar to the one of Theorem \ref{th-main1}, requiring now to look for intersections (in the phase-plane)
between the forward image of the positive $x$-semiaxis with a continuum of asymptotic solutions, which
can be easily found using the Conley-Wa\.zewski's method \cite{Con75,Waz47}. Let us mention that the idea of combining phase-plane analysis with the Conley-Wa\.zewski's method in order to obtain homoclinic solutions with complex behavior has been used in the recent papers \cite{BosDamPap15,DamPap->,EllZan13},
dealing however with different equations.

\begin{proof}[Sketch of the proof]
We use the same strategy as in the proof of Theorem \ref{th-main1} with minor changes.
Precisely, using the notation therein, by an iterative application of Proposition \ref{lemma-sap} we obtain the existence of 
$2^m$ disjoint intervals $[s_k^-,s_k^+]$, for $k=1,\ldots,2^m$, such that
$$
[s_k^-,s_k^+] \ni s \mapsto \varphi_{[\sigma_1',\omega_m']} (\gamma(s)) = \varphi_{[\tau_m',\omega_m']} \circ \varphi_{[\sigma_m',\tau_m']}
\circ \cdots \circ \varphi_{[\sigma_1',\tau_1']}(\gamma(s))  
$$
is a path contained in $\mathcal{R}(r,R)$ and joining the opposites sided
$\mathcal{R}_{\textnormal{left}}(r,R)$ and $\mathcal{R}_{\textnormal{right}}(r,R)$
(here $\omega_m'$ is fixed in such a way that $\tau_m' < \omega_m'$).

As for the dynamics on $[\omega_m',+\infty[\,$, instead, we use \cite[Lemma 5]{PapZan00}, ensuring (via an application of the Conley-Wa\.zewski's method) 
the existence of an unbounded continuum
$\Gamma_\mu^\infty \subset Q_4$, with $(0,0) \in \Gamma_\mu^\infty$, made by initial conditions $(v(\omega_m'),v'(\omega_m'))$ of solutions satisfying
$(v(t),v'(t)) \to (0,0)$ for $t \to +\infty$.  Notice that, according to \cite[Lemma 5]{PapZan00}, such a continuum exists since
$$
\int_{\omega_m'}^{+\infty} b(t)\,dt = \int_{h^{-1}(\omega_m')}^{+\infty} a(s)s^{N-1}\,ds = -\infty.
$$

At this point, the intersection lemma gives an intersection between the path $[s_k^-,s_k^+] \ni s \mapsto \varphi_{[\sigma_1',\omega_m']} (\gamma(s))$
and the continuum $\Gamma_\mu^\infty$: now, one intersection is just the origin $(0,0)$ (namely, the one for $ k = 1 $),
while all the other ones are non-trivial and give rise to the desired $2^m-1$ positive homoclinic solutions.
\end{proof}
\section*{Acknowledgements}\label{acknowledgements}
Alberto Boscaggin and Walter Dambrosio acknowledge the support of the
project ERC Advanced Grant 2013 n.~339958 ``Complex Patterns for Strongly Interacting Dynamical Systems - COMPAT''.
Walter Dambrosio is also supported by the P.R.I.N. Project ``Variational and perturbative aspects of nonlinear differential problems''.
This work is supported by the GNAMPA Project 2016
``Pro\-ble\-mi differenziali non lineari: esistenza, molteplicit\`{a} e propriet\`{a} qualitative delle soluzioni''.





\end{document}